\documentclass[12pt]{article}
\usepackage{amsmath,amssymb,amsfonts}
\usepackage{xcolor}
\usepackage{graphicx}
\usepackage{graphics}
\usepackage{hyperref}
\usepackage{amsfonts}
\usepackage{amsmath,amssymb,amsthm,latexsym,graphicx,graphics}
\usepackage{amsthm}
\usepackage{multicol}
\usepackage{color}
\usepackage{url,hyperref}
\usepackage{euscript}
\usepackage{wrapfig}
\usepackage{caption}
\usepackage{url,hyperref}
\usepackage{euscript,color}

\newtheorem{theorem}{Theorem}

\newtheorem{definition}[theorem]{Definition}

\newtheorem{lemma}[theorem]{Lemma}

\newtheorem{proposition}[theorem]{Proposition}
\newtheorem{remark}[theorem]{Remark}

\begin{document}
	
	\title{A split special Lagrangian calibration \\ associated with frame vorticity}
	\author{Marcos Salvai \thanks{%
			This work was supported by Consejo Nacional de Investigaciones Cient\'{\i}ficas
			y T\'ecnicas and Secre\-tar\'{\i}a de Ciencia y T\'ecnica de la Universidad
			Nacional de C\'ordoba.}}
	\date { }
	\maketitle
	
	\begin{abstract}
	Let $M$ be an oriented three-dimensional Riemannian manifold. We define a
notion of vorticity of local sections of the bundle $SO\left(
M\right) \rightarrow M$ of all its positively oriented orthonormal tangent
frames. When $M$ is a space form, we relate the concept to a
suitable invariant split pseudo-Riemannian metric on Iso$_{o}\left( M\right)
\cong SO\left( M\right) $: A local section has positive vorticity if and only
if it determines a space-like submanifold. In the Euclidean case we find
explicit homologically volume maximizing sections using a split special
Lagrangian calibration. We introduce the concept of optimal frame vorticity and
give an optimal screwed global section for the three-sphere. We prove that
it is also homologically volume maximizing (now using a common one-point
split calibration). Besides, we show that no optimal section can exist in
the Euclidean and hyperbolic cases.
\end{abstract}

\noindent \textsl{Keywords:} Special Lagrangian calibration, frame vorticity, split bi-invariant metric

\medskip

\noindent \textsl{Mathematics Subject Classification 2020:} 49Q20, 53C30, 53C38, 53D12, 58J60

\section{Introduction}

In this article we obtain several results on the screwness of sections of
the orthonormal frame bundle of a three-dimensional space form. We\
highlight the particular one involving a split special Lagrangian
calibration, due to the relevance of this technique. Particularly for
curved spaces with signature, although there have been important
applications (for instance, in relation with optimal transport), we feel
that concrete,
natural examples could be welcome.

Pseudo-Riemannian geometry is often the appropriate setting when dealing
with manifolds possessing two qualitatively different types of tangent
vectors (and a third borderline type), the paramount example being
Lorentzian geometry in relativity, reflecting the distinction between
space-like or time-like curves of events. On the manifold of all rigid
transformations of a three-dimensional space form $M$ we distinguish curves
(that is, motions of $M$) that describe, at each instant, positive or
negative screws. We consider on it a pseudo-Riemannian metric of signature $%
\left( 3,3\right) $ which accounts for this dichotomy.

The importance of extrema in mathematics cannot be overemphasized.
Calibrations detect submanifolds of minimum or maximum volume in a homology
class. The method is powerful because
of the global nature of its results. One can see the history of calibrations
in the book \cite{MorganB}. They grew impressively in strength in the
celebrated paper \cite{harvey-lawson} by Harvey and Lawson of 1982; see also 
\cite{HarveyBook}. Calibrations provided substantial achievements in volume
minimization of submanifolds of Euclidean space (see for instance \cite%
{DadokHarvey,IKSe, MorganRn, XYZe}), and also of several Riemannian manifolds, see for
instance \cite{GZ,gluck-morgan-mackenzie, GVm, HKm,Lm,morgan-ziller, Sm} (the first one
inaugurated the much pursued search for the best among structures of a
certain type on a Riemannian manifold).

In 1989 Mealy introduced in \cite{Mealy} calibrations on pseudo-Riemannian
manifolds. The split special Lagrangian calibrations were rediscovered by
Warren \cite{Warren}, with a new approach set up by Hitchin in \cite{Hi}.
Warren applied them to the volume maximization problem of the special
Lagrangian submanifolds of a pseudo-Euclidean space (see also \cite{HarveyA}%
). For curved manifolds, they were also useful for instance in relation to
optimal transport \cite{warren-ot} or foliations \cite{GSpams}.

\begin{definition}
	\label{hvm}Let $N$ be a pseudo-Riemannian manifold of dimension $2m$ with
	signature $\left( m,m\right) $ and let $M$ be an oriented space-like
	submanifold of $N$ of dimension $m$. One says that $M$ is \textbf{volume
		maximizing} in $N$ if for any open subset $U$ of $M$ with compact closure
	and smooth border $\partial U$, one has that 
	\begin{equation*}
		\operatorname{vol}_{m}\left( U\right) \geq \operatorname{vol}_{m}\left( V\right) 
	\end{equation*}%
	for any space-like submanifold $V$ of $N$ of dimension $m$ with compact
	closure and $\partial V=\partial U$.
		Moreover, $M$ is said to be \textbf{homologically volume maximizing} in $N$ if
	in addition $V$ is required to be homologous to $U$.
\end{definition}

\bigskip

Let $\mathfrak{so}_{3}=\left\{ A\in \mathbb{R}^{3\times 3}\mid
A^{T}=-A\right\} $ be the Lie algebra of $SO_{3}$. There is a linear
isomorphism $\,$%
\begin{equation}
	C:\mathbb{R}^{3}\rightarrow \mathfrak{so}_{3}\text{,\ \ \ \ \ }\xi \mapsto
	C_{\xi }\text{,\ \ \ \ \ \ where\ \ \ \ \ \ \ }C_{\xi }\left( y\right) =\xi
	\times y\text{.}  \label{Cruz}
\end{equation}%
Note that for $x\neq 0$, $\exp \left( C_{x}\right) $ is the rotation through
the angle $\left\vert x\right\vert $ around the oriented axis spanned by $x$.

Let $\mathbb{R}^{3}\rtimes SO_{3}$ be the group of direct (that is,
orientation preserving) isometries of Euclidean space $\mathbb{R}^{3}$, the
multiplication being given by $\left( x,A\right) \left( y,B\right) =\left(
x+Ay,AB\right) $. We endow it with the left invariant pseudo-Riemannian
metric of signature $\left( 3,3\right) $ given at the identity $\left(
0,I_{3}\right) $ by 
\begin{equation}
	\left\langle \left( x,C_{\xi }\right) ,\left( y,C_{\eta }\right)
	\right\rangle =\tfrac{1}{4}\left( \left\langle x,\eta \right\rangle
	+\left\langle y,\xi \right\rangle \right) \label{metricKappaCero}
\end{equation}%
($I_{k}$ is the $\left( k\times k\right) $-identity matrix). We
show below, in Section \ref{GeomSig}, that the metric is actually bi-invariant.

For a vector $X$ tangent to a pseudo-Riemannian manifold we denote $%
\left\Vert X\right\Vert =\left\langle X,X\right\rangle $ and $\left\vert
X\right\vert =\sqrt{\left\langle X,X\right\rangle }$, provided that the
integrand is nonnegative. A map $F$ from $\mathbb{R}^{3}$ to a group is said
to be odd if $F\left( -x\right) =F\left( x\right) ^{-1}$ for all $x\in 
\mathbb{R}^{3}$. By smooth we mean of class $\mathcal{C}^{\infty }$.

We can now state one of the main results of the article.

	\begin{theorem}
	\label{main}Given $c>0$, let 
	\begin{equation}
		\ell :\mathbb{R}\rightarrow \mathbb{R}\text{, \ \ \ \ \ \ \ \ }\ell \left(
		r\right) =c\left( r-\sin r\right) ^{1/3}\text{,}  \label{ell}
	\end{equation}%
	which is a strictly increasing odd smooth function with $\ell ^{\prime
	}\left( 0\right) >0$. Then the map 
	\begin{equation}
		\Phi :\mathbb{R}^{3}\rightarrow \mathbb{R}^{3}\rtimes SO_{3}\text{,\ \ \ \ \
			\ }\Phi \left( ru\right) =\left( \ell \left( r\right) u,\exp \left(
		C_{ru}\right) \right) \text{,}  \label{Phi}
	\end{equation}%
	with $\left\vert u\right\vert =1$ and $r\in \mathbb{R}$, is well defined,
	smooth and odd. Its restrictions to the open ball $\ \left\{ v\in \mathbb{R}%
	^{3}\mid \left\vert v\right\vert <\pi \right\} $ and to the open spherical
	shells 
	\begin{equation*}
		\left\{ v\in \mathbb{R}^{3}\mid 2k\pi <\left\vert v\right\vert <\left(
		2k+1\right) \pi \right\} \text{,}
	\end{equation*}%
	for $k\in \mathbb{N}$, determine homologically volume maximizing space-like
	submanifolds of $\mathbb{R}^{3}\rtimes SO_{3}$.
	\end{theorem}

\medskip

In Section \ref{SCali} we give the proof of this theorem, passing to the
universal covering, using a split special Lagrangian calibration. The
submanifolds of the statement, being calibrated, are in particular maximal,
that is, they have zero mean curvature. We do not know properties of general
solutions of the equation of Monge-Amp\`{e}re type for maximal space-like
submanifolds, e.g. whether they can project onto $\mathbb{R}^{3}$; 
we refer to the
recent survey \cite{Yuan}.

In Section \ref{GeomSig} we present the geometric significance of the
theorem, namely, that the given submanifolds are maximally screwed, in a
certain sense. We go beyond inquiring further about related subjects. Let $M$
be an oriented three-dimensional Riemannian manifold $M$ and let $SO\left(
M\right) \rightarrow M$ be the bundle of its positively oriented orthonormal
frames. First, we introduce the concept of vorticity of a section of
this bundle. Then, when $M$ is a space form, we relate its positivity with the fact that
the submanifold of $SO\left( M\right) $ determined by the section is
space-like for a split pseudo-Riemannian metric on it.

In Section \ref{ig} we study the intrinsic geometry of (a lift of) the
submanifold $\left. \Phi \right\vert _{B}$ in Theorem \ref{main}, where $B$
is the ball of radius $\pi $ in $\mathbb{R}^{3}$ centered at the origin. It is 
not complete and its metric completion is homeomorphic to the three-sphere.

In Section \ref{oh} we define the notion of optimal vorticity for sections of 
$SO\left( M\right) \rightarrow M$. We give an example for $M=S^{3}$ (which
turns out to be also homologically volume maximizing, using a common
one-point calibration) and show that no local section has that property for
Euclidean or hyperbolic space.

I would like to thank the anonymous reviewer for their comments.

\section{Calibrations and the proof of the theorem\label{SCali}}

We recall some definitions and results from Mealy in \cite{Mealy} (see also 
\cite{HarveyA}), which are analogous to those for Riemannian manifolds in 
\cite{harvey-lawson}. Let $N$ be a split pseudo-Riemannian manifold of
dimension $2m$ (by split we mean that it has signature $\left( m,m\right)$). 
For $q\in N$, an $m$-vector $\xi $ in $T_{q}N$ is said to be 
space-like if the subspace generated by $\xi $ is space-like (we are not
considering the induced indefinite inner product on $\Lambda ^{m}\left(
T_{q}N\right) $).

\begin{definition}
	Let $N$ be a split pseudo-Riemannian manifold of dimension $2m$. A closed $m$%
	-form $\psi $ on $N$ is called a \textbf{calibration} if $\psi _{q}\left( \xi
	\right) \geq \operatorname{vol}_{m}\left( \xi \right) $ for any space-like $m$-vector $%
	\xi $ in $T_{q}N$ with $\psi _{q}(\xi )>0$, for all $q\in N$.
	
	Let $\psi $ be a calibration on $N$ and let $M$ be an oriented space-like
	submanifold of $N$ of dimension $m$. Then $M$ is said to be \textbf{calibrated} 
	 by $\psi $ if $\psi _{q}\left( \xi \right) =\operatorname{vol}_{m}\left( \xi
	\right) $ for any positively oriented $m$-vector $\xi $ generating $T_{q}M$
	with $\psi _{q}(\xi )>0$, for all $q\in M$.
\end{definition}

\begin{theorem}
	\label{tfc}\emph{\cite{Mealy}} Let $N$ be a split pseudo-Riemannian manifold of
	dimension $2m$ and let $\psi $ be a calibration on $N$. If $M$ is an
	oriented space-like submanifold of $N$ of dimension $m$ which is calibrated
	by $\psi $, then $M$ is homologically volume maximizing in $N$.
\end{theorem}

Next we introduce the split special Lagrangian calibration on a split
Euclidean space, which appeared first in the work of Mealy \cite{Mealy}. We
consider the presentation of this calibration in null coordinates given in 
\cite{Warren}.

\begin{proposition}
	\label{warren}\emph{\cite{Warren}} Let $\mathbb{R}^{m,m}$ be $\mathbb{R} ^{m}\times 
	\mathbb{R}^{m}$ endowed with the split inner product whose associated square
	norm is $\left\Vert \left( x,y\right) \right\Vert =\left\langle
	x,y\right\rangle $, where $\left\langle .,.\right\rangle $ is the canonical
	inner product on $\mathbb{R}^{m}$. For any $C>0$, the $m$-form 
	\begin{equation*}
		\frac{1}{2}\left( Ce^{1}\wedge \dots \wedge e^{m}+\frac{1}{C}~\epsilon
		^{1}\wedge \dots \wedge \epsilon ^{m}\right)
	\end{equation*}
	is a calibration on $\mathbb{R}^{m,m}$ 
	(here $\left\{ e^{1},\dots,e^{m},\epsilon ^{1},\dots,\epsilon ^{m}\right\} $ is the
	dual of the canonical basis of $\mathbb{R}^{2m}$). Moreover, a
	space-like $m$-vector $\xi $ is calibrated by it if and only if
	\begin{equation}
		\epsilon ^{1}\wedge \dots \wedge \epsilon ^{m}\left( \xi \right)
		=C^{2}\,e^{1}\wedge \dots \wedge e^{m}\left( \xi \right) \text{,}
		\label{EqWarren}
	\end{equation}
	that is, $\xi $ is special Lagrangian.
\end{proposition}

\bigskip

To prove Theorem \ref{main} it will be convenient to pass to the universal
covering of $\mathbb{R}^{3}\rtimes SO_{3}$ and compute with quaternions.
Let $\mathbb{H}$ be the skew field of quaternions and let $S^{3}=\left\{
q\in \mathbb{H}\mid \left\vert q\right\vert =1\right\} $. We identify as
usual $\mathbb{R}^{4}=\mathbb{R}+\operatorname{Im}\left( \mathbb{H}\right) $ and
consider the canonical basis $\left\{ 1,i,j,k\right\} $. Let $\mathbb{R}%
^{3}\rtimes S^{3}$ be the product $\mathbb{R}^{3}\times S^{3}$ endowed with
the multiplication 
\begin{equation*}
	\left( x,p\right) \left( y,q\right) =\left( x+py\bar{p},pq\right) \text{,}
\end{equation*}%
where $\bar{q}=q^{-1}$. For further reference we compute
\begin{equation}
	\left( dL_{\left( x,p\right) }\right) _{\left( 0,1\right) }\left( \xi ,\eta
	\right)  =\left. \frac{d}{dt}\right\vert _{0}\left( x+pt\xi \bar{p}%
	,pe^{t\eta }\right) =\left( p\xi \bar{p},p\eta \right) \text{.}  \label{dL}
\end{equation}

We endow the Lie group $\mathbb{R}^{3}\rtimes S^{3}$ with the left invariant
split pseudo-Riemannian metric given at the identity $\left( 0,1\right) $ by%
\begin{equation}
	\left\langle \left( x,\xi \right) ,\left( y,\eta \right) \right\rangle =%
	\tfrac{1}{2}\left( \left\langle x,\eta \right\rangle +\left\langle y,\xi
	\right\rangle \right) \text{.}  \label{prmS3}
\end{equation}%
for $x,y\in \operatorname{Im}\left( \mathbb{H}\right) $ and $\xi ,\eta \in
T_{1}S^{3}=\operatorname{Im}\mathbb{H}$.

\smallskip 

Let $I:S^{3}\rightarrow SO_{3}$, $I\left( q\right) =I_{q}$, with $%
I_{q}\left( x\right) =qx\bar{q}$ for $x\in \operatorname{Im}\mathbb{H}\cong \mathbb{R%
}^{3}$, and let 
\begin{equation}
	\Pi :\mathbb{R}^{3}\rtimes S^{3}\rightarrow \mathbb{R}^{3}\rtimes SO_{3}%
	\text{,\ \ \ \ \ \ \ \ }\Pi \left( x,q\right) =\left( x,I_{q}\right) \text{.}
	\label{Pi}
\end{equation}%
Both are smooth two-to-one surjective Lie group morphisms.

\begin{proposition}
	\label{PiIso}The metric above on $\mathbb{R}^{3}\rtimes S^{3}$ is
	bi-invariant. In particular, we have the following rotational symmetry: the
	left action of $S^{3}$ on $\mathbb{R}^{3}\rtimes S^{3}$ given by%
	\begin{equation}
		\left( q,\left( x,p\right) \right) \mapsto L_{\left( 0,q\right) }R_{\left( 0,%
			\bar{q}\right) }\left( x,p\right) =\left( qx\bar{q},qp\bar{q}\right) =\left(
		I_{q}\left( x\right) ,qp\bar{q}\right)   \label{action}
	\end{equation}%
	is by isometries. Also, $\Pi $ is a local isometry if $\mathbb{R}^{3}\rtimes
	SO_{3}$ is endowed with the metric defined in \emph{(\ref{metricKappaCero})}.
\end{proposition}

\begin{proof}
	We prove first the last assertion. We compute $\left( dI\right)
	_{1}:T_{1}S^{3}=\operatorname{Im}\mathbb{H}\rightarrow \mathfrak{so}_{3}$: For $z\in 
	\operatorname{Im}\mathbb{H}$ we have 
	\begin{equation*}
		dI_{1}\left( \xi \right) \left( z\right) =\left. \frac{d}{dt}\right\vert
		_{0}e^{t\xi }ze^{-t\xi }=\xi z-z\xi =2\xi \times z=C_{2\xi }\left( z\right) 
		\text{.}
	\end{equation*}%
	Hence, 
	\begin{equation*}
		\left\Vert d\Pi _{\left( 0,1\right) }\left( x,\xi \right) \right\Vert
		=\left\Vert \left( x,C_{2\xi }\right) \right\Vert =\tfrac{1}{2}\left\langle
		2\xi ,\eta \right\rangle =\left\langle \xi ,\eta \right\rangle =\left\Vert
		\left( x,\xi \right) \right\Vert \text{.}
	\end{equation*}%
	Then, using the left invariance we have that $\Pi $ is a local isometry.
	
	In the paragraph containing expression (\ref{33}) below we show in
	particular that the metric on $\mathbb{R}%
	^{3}\rtimes SO_{3}$ (called $G_{0}$ there) defined in (\ref{metricKappaCero})  
	is bi-invariant. Now, the map $%
	d\Pi _{\left( 0,1\right) }$ is a linear isometric Lie algebra isomorphism.
	Since for connected Lie groups (as in our case) bi-invariance depends only
	on the inner product on the Lie algebra, the metric on $\mathbb{R}%
	^{3}\rtimes S^{3}$ is also bi-invariant. The second assertion is an
	immediate corollary of that property.
\end{proof}


We consider maps of the form%
\begin{equation*}
	\varphi :\mathbb{R}^{3}\rightarrow \mathbb{R}^{3}\rtimes S^{3}\text{,\ \ \ \
		\ \ }\varphi \left( ru\right) =\left( l\left( r\right) u,\exp \left( \theta
	\left( r\right) u\right) \right) \text{,}
\end{equation*}%
with $\left\vert u\right\vert =1$ and $r\in \mathbb{R}$, where $l:\mathbb{R}%
\rightarrow \mathbb{R}$ and $\theta :\mathbb{R}\rightarrow \mathbb{R}$ are
two smooth odd functions. We call the map $\varphi $ a \emph{screw-radial map%
} from $\mathbb{R}^{3}$ to $\mathbb{R}^{3}\rtimes S^{3}$.

From the following lemma one deduces conditions on $l$ and $\theta $ for
restrictions of $\varphi $ to open sets of $\mathbb{R}^{3}$ to be space-like
submanifolds of $\mathbb{R}^{3}\rtimes S^{3}$.

\begin{lemma}
	\label{space-like}The map $\varphi $ is well defined, smooth and odd. For $%
	r>0$ and $\left\vert u\right\vert =1$, the image of $\left( d\varphi \right)
	_{ru}$ is a space-like three-dimensional subspace of $T_{\varphi \left(
		ru\right) }\left( \mathbb{R}^{3}\rtimes S^{3}\right) $ if and only if 
	\begin{equation*}
		l^{\prime }\left( r\right) >0\text{,\ \ \ \ \ }\theta ^{\prime }\left(
		r\right) >0\text{ \ \ \ \ \ and \ \ \ \ }k\pi <\theta \left( r\right) <k\pi +%
		\frac{\pi }{2}
	\end{equation*}%
	for some $k\in \mathbb{N}\cup \left\{ 0\right\} $. For $r=0$, the same
	holds, but dropping the third condition.
\end{lemma}

\begin{proof}
	We have $\varphi \left( 0\right) =\left( 0,1\right) $. For $0\neq v\in 
	\mathbb{R}^{3}$ we write $v=r\left( v/r\right) $, with $r=\left\vert
	v\right\vert $, and check that $\varphi \left( \left( -r\right) \left(
	-v/r\right) \right) =\varphi \left( r\left( v/r\right) \right) $, hence $%
	\varphi $ is well defined. The map $\varphi $ is odd since $e^{sw}$ commutes
	with $w$ for all $w\in \operatorname{Im}\left( \mathbb{H}\right) $ and $s\in \mathbb{%
		R}$.
	Since $l$ is smooth and odd, the function $\mathbb{R}\rightarrow \mathbb{R}$
	given by $s\mapsto l\left( \sqrt{\left\vert s\right\vert }\right) /\sqrt{%
		\left\vert s\right\vert }$, $0\mapsto l^{\prime }\left( 0\right) $, is
	smooth; similarly for $\theta $. Hence, $\varphi $ is smooth.
	
	Notice that $\varphi $ is equivariant by the left actions of $S^{3}$ on $%
	\mathbb{R}^{3}=\operatorname{Im}\left( \mathbb{H}\right) $ given by $\left(
	q,x\right) \mapsto qx\bar{q}=I_{q}\left( x\right) $ and that on $\mathbb{R}%
	^{3}\rtimes S^{3}$ in (\ref{action}). The first action is clearly isometric
	and transitive on the two-sphere and the second one is isometric by
	Proposition \ref{PiIso}. Thus, it suffices to analyze when $d\varphi _{ri}$
	is one-to one and its image is space-like for $r\geq 0$. We consider first
	the case $r>0$ and then study the limit for $r\rightarrow 0^{+}$. We compute%
	\begin{equation}
		d\varphi _{ri}\left( i\right) =\left. \frac{d}{dt}\right\vert _{r}\left(
		l\left( t\right) i,e^{\theta \left( t\right) i}\right) =\left( l^{\prime
		}\left( r\right) i,\theta ^{\prime }\left( r\right) ie^{\theta \left(
			r\right) i}\right)   \label{dFi}
	\end{equation}
	
	Now let $z$ be a unit vector in span$\left\{ j,k\right\} $ and let $\alpha
	\left( t\right) =\cos \left( t/r\right) i+\sin \left( t/r\right) z$. This
	curve satisfies 
	\begin{equation*}
		r\alpha \left( 0\right) =ri\text{,\ \ \ \ }r\alpha ^{\prime }\left( 0\right)
		=z\text{\ \ \ \ \ \ and \ \ \ \ \ }e^{\theta \left( r\right) \alpha \left(
			t\right) }=\cos \theta \left( r\right) +\sin \theta \left( r\right) \alpha
		\left( t\right) \text{.}
	\end{equation*}%
	Hence, 
	\begin{eqnarray}
		d\varphi _{ri}\left( z\right) &=&\left. \frac{d}{dt}\right\vert _{0}\varphi
		\left( r\alpha \left( t\right) \right) =\left. \frac{d}{dt}\right\vert
		_{0}\left( l\left( r\right) \alpha \left( t\right) ,e^{\theta \left(
			r\right) \alpha \left( t\right) }\right) \\ \label{dFu}
		&=&\frac{1}{r}\left( l\left( r\right)
		z,\sin \theta \left( r\right) ~z\right) \text{.}  \notag
	\end{eqnarray}%
	Using (\ref{dL}), we have that $\left( dL_{\varphi \left( ri\right) }\right)
	_{\left( 0,1\right) }\left( \xi ,\eta \right) =\left( e^{\theta \left(
		r\right) i}\xi e^{-\theta \left( r\right) i},e^{\theta \left( r\right)
		i}\eta \right) $. We apply the inverse of this linear map to (\ref{dFi}) and
	(\ref{dFu}), obtaining 
	\begin{equation}
		\left( l^{\prime }\left( r\right) i,\theta ^{\prime }\left( r\right)
		i\right) \text{\ \ \ \ and\ \ \ \ }\frac{1}{r}\left( l\left(
		r\right) e^{-\theta \left( r\right) i}ze^{\theta \left( r\right) i},\sin
		\left( \theta \left( r\right) \right) ~e^{-\theta \left( r\right) i}z\right) 
		\text{,}  \label{tpa}
	\end{equation}%
	respectively. Since the metric is left invariant and $\left\langle
	ze^{\theta \left( r\right) i},z\right\rangle =\left\langle e^{\theta \left(
		r\right) i},1\right\rangle =\cos \left( \theta \left( r\right) \right) $, we
	have that 
	\begin{equation}
		\left\Vert d\varphi _{ri}\left( i\right) \right\Vert =\left\langle l^{\prime
		}\left( r\right) i,\theta ^{\prime }\left( r\right) i\right\rangle
		=l^{\prime }\left( r\right) \theta ^{\prime }\left( r\right) \text{,}
		\label{dphii}
	\end{equation}%
	\begin{eqnarray}
		\left\Vert d\varphi _{ri}\left( z\right) \right\Vert  &=&\frac{1}{r^{2}}%
		\left\langle l\left( r\right) e^{-\theta \left( r\right) i}ze^{\theta \left(
			r\right) i},\sin \left( \theta \left( r\right) \right) e^{-\theta \left(
			r\right) i}z\right\rangle   \notag \\
		&=&\frac{1}{r^{2}}~l\left( r\right) \sin \left( \theta \left( r\right)
		\right) \left\langle ze^{\theta \left( r\right) i},z\right\rangle 
		\label{dphiu} \\
		&=&\frac{1}{2r^{2}}~l\left( r\right) \sin \left( 2\theta \left( r\right)
		\right) \text{.}  \notag
	\end{eqnarray}%
	For $z\perp i$, $e^{si}z$ and $e^{-si}ze^{si}$ are orthogonal to $i$. Hence (%
	\ref{prmS3}) yields $d\varphi _{ri}\left( i\right) \perp 
	d\varphi_{ri}\left( z\right)$.
		
	Consequently, for $r>0$, $d\varphi _{ri}$ is injective and its image is
	space-like if and only if $\theta $, $\theta ^{\prime }$ and $l^{\prime }$
	evaluated at $r$ are as stated. Taking the limit for $r\rightarrow 0$, we
	have that $\left\langle d\varphi _{0}\left( i\right) ,d\varphi _{0}\left(
	z\right) \right\rangle =0$ and $\left\Vert d\varphi _{0}\left( i\right)
	\right\Vert =l^{\prime }\left( 0\right) \theta ^{\prime }\left( 0\right)
	=\left\Vert d\varphi _{0}\left( z\right) \right\Vert $. Thus, the last
	assertion follows.
\end{proof}

\bigskip

The lemma shows that $\theta ^{\prime }>0$ is necessary for the screw-radial
map to be space-like. So we can reparametrize and in the next theorem we
consider $\theta \left( r\right) =r/2$.

Let $\mathbb{R}^{3}\rtimes S^{3}$ be endowed with the left invariant split
pseudo-Riemannian metric defined in (\ref{prmS3}) and let $\ell :\mathbb{R}%
\rightarrow \mathbb{R}$ be as in (\ref{ell}). The following is the lifted
version of Theorem \ref{main}.

\begin{theorem}
	\label{mainLIft}The map 
	\begin{equation}
		\phi :\mathbb{R}^{3}\cong \operatorname{Im}\mathbb{H}\rightarrow \mathbb{R}%
		^{3}\rtimes S^{3}\text{,\ \ \ \ \ \ }\phi \left( ru\right) =\left( \ell
		\left( r\right) u,\exp \left( ru/2\right) \right) \text{,}  \label{phi}
	\end{equation}%
	with $\left\vert u\right\vert =1$ and $r\in \mathbb{R}$, is well defined,
	smooth and odd. Its restrictions to the open ball $\ \left\{ v\in \mathbb{R}%
	^{3}\mid \left\vert v\right\vert <\pi \right\} $ and to the open spherical
	shells 
	\begin{equation*}
		\left\{ v\in \mathbb{R}^{3}\mid 2k\pi <\left\vert v\right\vert <\left(
		2k+1\right) \pi \right\} \text{,}
	\end{equation*}%
	for $k\in \mathbb{N}$, determine homologically volume maximizing space-like
	submanifolds of $\mathbb{R}^{3}\rtimes S^{3}$.
\end{theorem}

\begin{proof}
	Using Lemma \ref{space-like}\ with $l=\ell $ and $\theta \left( r\right) =r/2
	$, one shows that those restrictions of $\phi $ are space-like submanifolds.
	Indeed, on the one hand, $\theta ^{\prime }=1/2>0$ and $2k\pi <r<\left(
	2k+1\right) \pi $ implies that $k\pi <\frac{r}{2}<k\pi +\frac{\pi }{2}$ (for 
	$k\in \mathbb{N}\cup \left\{ 0\right\} $); on the other hand, 
	\begin{equation}
		\ell ^{\prime }\left( r\right) =\frac{c}{3}\frac{1-\cos r}{\left( r-\sin
			r\right) ^{2/3}}  \label{lPrime}
	\end{equation}%
	is positive for $0<r\neq 2k\pi $ ($k\in \mathbb{N}$) and also $\ell ^{\prime
	}\left( 0\right) >0$. In particular, $\ell $ is strictly increasing and so $%
	\varphi $ is injective.
	
	Now we prove that the given restrictions of $\varphi $ are homologically
	volume maximizing. Let $\pi _{1}$ and $\pi _{2}$ be the canonical
	projections of $\mathbb{R}^{3}\rtimes S^{3}$ onto the first and second
	factor, and let $\omega ^{1}$ and $\omega ^{2}$ be the canonical volume
	forms of $\mathbb{R}^{3}$ and $S^{3}$, respectively. Define the $3$-form $%
	\omega $ on $\mathbb{R}^{3}\rtimes S^{3}$ by 
	\begin{equation}
		\omega =\frac{1}{2}\left( C\pi _{1}^{\ast }\omega ^{1}+\frac{1}{C}\,\pi
		_{2}^{\ast }\omega ^{2}\right) \text{,}  \label{omegaR3S3}
	\end{equation}%
	for some positive constant $C$ to be determined later. Let us see that $%
	\omega $ is a calibration. Clearly, $\omega $ is closed. We call $\left\{
	e^{1},e^{2},e^{3},\epsilon ^{1},\epsilon
	^{2},\epsilon ^{3}\right\} $ the dual basis of the juxtaposition of $\left\{ \left( i,0\right)
	,\left( j,0\right) ,\left( k,0\right) \right\} $ and $\left\{ \left(
	0,i\right) ,\left( 0,j\right) ,\left( 0,j\right) \right\} $
	(sometimes, we will abuse the notation omitting the zeros). We have that 
	\begin{equation*}
		\omega _{\left( 0,1\right) }=\frac{1}{2}\left( Ce^{1}\wedge e^{2}\wedge
		e^{3}+\frac{1}{C}\,\epsilon ^{1}\wedge \epsilon ^{2}\wedge \epsilon
		^{3}\right) \text{,}
	\end{equation*}%
	which is a calibration on $T_{\left( 0,1\right) }\left( \mathbb{R}%
	^{3}\rtimes S^{3}\right) =\mathbb{R}^{3}\times \mathbb{R}^{3}$ by
	Proposition \ref{warren}. Since both $\omega $ and the split metric are left
	invariant, $\omega $ is a calibration on $\mathbb{R}^{3}\rtimes S^{3}$.
	
	Next we verify that $\omega $ calibrates $\phi $. By the rotational symmetry
	we have already used in the proof of Lemma \ref{space-like}, if suffices to
	show that the image of $d\phi _{ri}$ is calibrated by $\omega _{\phi \left(
		ri\right) }$, or equivalently, by invariance, that $\left( L_{\phi \left(
		ri\right) }\right) ^{\ast }d\phi _{ri}\left( \mathbb{R}^{3}\right) $ is
	calibrated by $\omega _{\left( 0,1\right) }$. For any $r>0$ and any unit
	vector $z\perp i$ in $\operatorname{Im}\mathbb{H}$, by (\ref{tpa})\ with $l=\ell $
	and $\theta =r/2$, we have 
	\begin{eqnarray*}
		\left( L_{\phi \left( ri\right) }\right) ^{\ast }d\phi _{ri}\left( i\right) 
		&=&\left( \ell ^{\prime }\left( r\right) i,i/2\right)  \\
		\left( L_{\phi \left( ri\right) }\right) ^{\ast }d\phi _{ri}\left( z\right) 
		&=&\frac{1}{r}\left( \ell \left( r\right) e^{-ri/2}ze^{ri/2},\sin \left(
		r/2\right) ~e^{-ri/2}z\right) \text{.}
	\end{eqnarray*}%
	We compute 
	\begin{equation*}
		\begin{tabular}{l}
			$e^{1}\wedge e^{2}\wedge e^{3}\left( \ell ^{\prime }\left( r\right) i,\dfrac{%
				1}{r}\,\ell \left( r\right) e^{-ri/2}je^{ri/2},\dfrac{1}{r}\,\ell \left(
			r\right) e^{-ri/2}ke^{ri/2}\right) =$ \\ 
			\\ 
			$\ \ \ =\dfrac{1}{r^{2}}\,\ell \left( r\right) ^{2}\ell ^{\prime }\left(
			r\right) e^{1}\wedge e^{2}\wedge e^{3}\left( i,j,k\right) =\dfrac{1}{3r^{2}}\dfrac{d}{dr}\left( \ell \left( r\right)
			^{3}\right) \text{,}$ 
			\end{tabular}%
	\end{equation*}%
	since $i=e^{-ri/2}ie^{ri/2}$ and conjugation by $e^{-ri/2}$ preserves $%
	\omega _{0}^{1}$. Also, 
	\begin{equation*}
		\begin{tabular}{l}
			$\epsilon ^{1}\wedge \epsilon ^{2}\wedge \epsilon ^{3}\left( \dfrac{1}{2}\,i,%
			\dfrac{1}{r}\sin \left( r/2\right) ~e^{-ri/2}j,\dfrac{1}{r}\sin \left(
			r/2\right) ~e^{-ri/2}k\right) =$ \\ 
			\\ 
			$\ \ \ =\dfrac{\sin ^{2}\left( r/2\right) }{2r^{2}}~\epsilon ^{1}\wedge
			\epsilon ^{2}\wedge \epsilon ^{3}\left( i,j,k\right)=\dfrac{1-\cos r}{4r^{2}} $ 
		\end{tabular}%
	\end{equation*}%
	(on $i^{\bot }$, multiplication by $e^{is}$ is the rotation through the
	angle $s$). Then equation (\ref{EqWarren}) is satisfied in our case if and
	only if 
	\begin{equation*}
		\frac{1-\cos r}{4}=\frac{C^{2}}{3}\frac{d}{dr}\left( \ell \left( r\right)
		^{3}\right) 
	\end{equation*}%
	(the case $r=0$ is dealt with using continuity). Since $\ell \left( 0\right)
	=0$, this amounts to $\ell \left( r\right) ^{3}=\frac{3}{4C^{2}}\left(
	r-\sin r\right) $. So, we can choose $C>0$ with $4c^{3}C^{2}=3$ and the
	theorem follows.
\end{proof}

\medskip

Now, Theorem \ref{main} is a corollary of the previous theorem:

\begin{proof}[Proof of Theorem \protect\ref{main}]
	Let $\omega ^{1}$ and $\omega ^{2}$ be, as above, the canonical volume forms
	of $\mathbb{R}^{3}$ and $S^{3}$, respectively. Since odd dimensional real
	projective spaces are orientable, there exists a $3$-form $\varpi ^{2}$ on $%
	SO_{3}$ such that $I^{\ast }\varpi ^{2}=\omega ^{2}$.
	
	Let $\mathit{p}_{1}$, $\mathit{p}_{2}$ be the canonical projections of $%
	\mathbb{R}^{3}\times SO_{3}$ onto the first and second factors,
	respectively, and let $\Omega =\frac{1}{2}\left( C\mathit{p}_{1}^{\ast
	}\omega ^{1}+\frac{1}{C}~\mathit{p}_{2}^{\ast }\varpi ^{2}\right) $ on $%
	\mathbb{R}^{3}\times SO_{3}$. Then, $\Pi ^{\ast }\Omega =\omega $, where $%
	\omega $ is the calibration form in (\ref{omegaR3S3}), and thus $\Omega $ is
	a calibration. We also have that $\Pi \circ \phi =\Phi $, where $\phi $ is
	as in (\ref{phi}). Since $\Pi $ is a local isometry and $\phi $ is calibrated by 
	$\omega $, then $\Phi $ is calibrated by $\Omega $.
\end{proof}

\section{Geometric significance: Frame vorticity\label{GeomSig}}

The curl of vector fields on a three-dimensional Riemannian manifold is a
central concept in some areas of mathematics. We introduce below the
concept of vorticity of sections of orthonormal frame bundles. As a
motivation, we comment on the simplest screw-radial map from $\mathbb{R}^{3}$
to $\mathbb{R}^{3}\times SO_{3}$.

Let $SO\left( \mathbb{R}^{3}\right) =\mathbb{R}^{3}\times SO_{3}$ be the
positively oriented orthonormal frame bundle of $\mathbb{R}^{3}$. We
consider the section%
\begin{equation*}
	b_{0}:\mathbb{R}^{3}\rightarrow SO\left( \mathbb{R}^{3}\right) \text{,\ \ \
		\ \ \ \ }b_{0}\left( x\right) =\left( x,R_{x}\right) \text{,}
\end{equation*}%
where $R_{x}$ is the rotation through the angle $\left\vert x\right\vert $
around the line $\mathbb{R}x$, more precisely, $R_{0}=$ id and for $x\neq 0$%
, $R_{x}$ is the linear map defined by $R_{x}\left( x\right) =x$ and 
\begin{equation}
	R_{x}\left( y\right) =\cos \left(\left\vert x\right\vert\right) y+\sin \left(\left\vert
	x\right\vert\right) \left(x/{\left\vert x\right\vert }\right)\times y  \label{Rx}
\end{equation}%
for $y\perp x$. Notice that $\left. \frac{d}{dt}\right\vert _{0}R_{tx}=C_{x}$%
, with $C$ as in (\ref{Cruz}), and so the tangent space at $\left(
0,I_{3}\right) $ of the submanifold of $SO\left( \mathbb{R}^{3}\right) $
determined by $b_{0}$ is $\left\{ \left( x,C_{x}\right) \mid x\in \mathbb{R}%
^{3}\right\} $.

The section $b_{o}$ may be described informally as follows: Moving away from
the origin in one direction $x$ entails rotating through an angle $%
\left\vert x\right\vert $ around the oriented line determined by $x$. In
broad terms, we want to discern to what extent a local section of $SO\left( 
\mathbb{R}^{3}\right) \rightarrow \mathbb{R}^{3}$, at each point of $\mathbb{%
	R}^{3}$, resembles $b_{0}$ near the origin. For the sake of generality we
study the problem in a wider context.

Let $M$ be an oriented three-dimensional Riemannian manifold. For an
orthonormal set $\left\{ u,v\right\} $ in $T_{p}M$, let $u\times v$ be the
unique $w\in T_{p}M$ such that $\left\{ u,v,w\right\} $ is a positively
orthonormal basis. The bilinear extension gives a well defined cross product 
$\times $ on $T_{p}M$ depending smoothly on $p$. For $x\in T_{p}M$ we denote 
$C_{x}\left( y\right) =x\times y$, which defines a skew symmetric operator
on $T_{p}M$. Also, all such operators have that form.

Fix a point $o$ in $M$ and let $SO\left( M\right) $ be the bundle of
positively oriented orthonormal frames of $M$, that is,%
\begin{equation*}
	SO\left( M\right) =\left\{ b:T_{o}M\rightarrow T_{p}M\mid b\text{ is a
		direct linear isometry, }p\in M\right\} \text{.}
\end{equation*}%
This is a principal fiber bundle over $M$ with typical fiber $SO_{3}$. The
definition differs unessentially from the usual one, with $\mathbb{R}^{3}$
instead of $T_{o}M$. On the one hand, this will allow us to think of the
elements of $SO\left( M\right) $ as positions of a body in $M$ with
reference state at $o$. On the other hand, when $M$ is a space form, this
will induce a simple identification between $SO\left( M\right) $ and the
group of direct isometries of $M$.

Let $b$ be a smooth section of $SO\left( M\right) \rightarrow M$. If $x\in
T_{p}M$, then $\nabla _{x}b:T_{o}M\rightarrow T_{p}M$ is well defined, as
usual, by%
\begin{equation*}
	\left( \nabla _{x}b\right) \left( y\right) =\left. \frac{D}{dt}\right\vert
	_{0}b\left( \alpha \left( t\right) \right) \left( y\right) \text{,}
\end{equation*}%
where $\alpha $ is any curve in $M$ with $\alpha \left( 0\right) =p$ and $%
\alpha ^{\prime }\left( 0\right) =x$.

	\begin{proposition}
	\label{prompt}Let $b$ be a smooth section of $SO\left( M\right) \rightarrow
	M $. For each vector field $X$ on $M$ there exists a unique vector field $%
	X^{b} $ on $M$ such that 
	$\nabla _{X}b=C_{X^{b}}\circ b$.
	\end{proposition}

The map $X\mapsto X^{b}$ determines a $\left( 1,1\right) $-tensor field on $%
M $ prompting the following definition ($T^1 M$ denotes the unit tangent bundle of $M$).

\begin{definition} The \textbf{vorticity} of a smooth section $b$ of $SO\left( M\right) \rightarrow M$ is 
	the function
	\begin{equation*}
		h^b:T^1M \rightarrow \mathbb R  \text{, \ \ \ \ \ } h^b\left(x\right)=\left\langle
		X^{b},X\right\rangle\text{.}
		\end{equation*}%
\end{definition}

Positive vorticity of $b$ means, informally, that at each point of $M$, when moving according to $%
b $ in any direction, this direction forms an angle smaller than $\pi /2$
with the axis of rotation. 

\begin{proof}[Proof of Proposition \protect\ref{prompt}]
	Let $p\in M$ and let $\alpha $ be a smooth curve on $M$ such that $\alpha
	\left( 0\right) =p$ and $\alpha ^{\prime }\left( 0\right) =X_{p}$. Given $%
	y\in T_{o}M$, $t\mapsto b\left( \alpha \left( t\right) \right) \left(
	y\right) $ is a unit vector field along $\alpha $ taking the value $b\left(
	p\right) \left( y\right) $ at $t=0$. Hence, 
	\begin{equation*}
		\left( \nabla _{X_{p}}b\right) \left( y\right) =\left. \frac{D}{dt}%
		\right\vert _{0}b\left( \alpha \left( t\right) \right) \left( y\right) \perp
		b\left( p\right) \left( y\right) \text{.}
	\end{equation*}%
	Putting $z=b\left( p\right) y$, we have that $\left\langle \left( \nabla
	_{X_{p}}b\right) \left( b\left( p\right) ^{-1}z\right) ,z\right\rangle $ for
	all $z\in T_{p}M$. Then, $\left( \nabla _{X_{p}}b\right) \circ b\left(
	p\right) ^{-1}$ is a skew symmetric operator on $T_{p}M$ and so it can be
	realized as $C_{w}$ for a unique $w\in T_{p}M$, which we call $X_{p}^{b}$.
\end{proof}

\bigskip

Now we turn our attention to the three-dimensional space form $M_{\kappa }$
of constant sectional curvature $\kappa =0,1$ or $-1$, that is, $M_{0}=%
\mathbb{R}^{3}$, $M_{1}$ is the sphere $S^{3}$ and $M_{-1}$ is hyperbolic
space $H^{3}$.

Let $\left\{ e_{0},e_{1},e_{2},e_{3}\right\} $ be the canonical basis of $%
\mathbb{R}^{4}$ and consider the inner product given by $\left\langle
x,y\right\rangle _{\kappa }=\kappa
x_{0}y_{0}+x_{1}y_{1}+x_{2}y_{2}+x_{3}y_{3}$. For $\kappa =\pm 1$, $%
M_{\kappa }$ is the connected component of $e_{0}$ of $\left\{ x\in \mathbb{R%
}^{4}\mid \left\langle x,x\right\rangle _{\kappa }=\kappa \right\} $, with
the induced (Riemannian) metric. To handle the three cases simultaneously,
sometimes it will be convenient to identify $\mathbb{R}^{3}\cong \left\{
x\in \mathbb{R}^{4}\mid x_{0}=1\right\} $. For any $p\in M_{\kappa }$, the
orientation on $T_{p}M_{\kappa }$ is given, as usual, by declaring an
oriented basis $\left\{ u,v,w\right\} $ positive if $\left\{ p,u,v,w\right\} 
$ is a positive basis of $\mathbb{R}^{4}$. In particular, $e_{1}\times
e_{2}=e_{3}$ at $e_{0}\in M_{\kappa }$.

We denote by $G_{\kappa }=$ Iso$_{o}\left( M_{\kappa }\right) $, the group of
direct isometries of $M_{\kappa }$. For $\kappa =1,-1$, it coincides with
the identity component of the automorphism group of the inner product $%
\left\langle ,\right\rangle _{1}$ or $\left\langle ,\right\rangle _{-1}$
(that is, $SO_{4}$ or $O_{o}\left( 1,3\right) $), respectively. With the
identification $\mathbb{R}^{3}\cong e_{0}+\mathbb{R}^{3}$, we have%
\begin{equation*}
	G_{0}=\left\{ \left( 
	\begin{array}{cc}
		1 & 0 \\ 
		a & A%
	\end{array}
	\right) \mid a\in \mathbb{R}^{3}\text{, }A\in SO_{3}\right\} \cong \mathbb{R}
	^{3}\rtimes SO_{3}\text{.}
\end{equation*}%
The Lie algebra of $G_{\kappa }$ is $\mathfrak{g}_{\kappa }=\left\{
Z_{\kappa }\left( x,\xi \right) \mid x,\xi \in \mathbb{R}^{3}\right\} $,
where%
\begin{equation*}
	Z_{\kappa }\left( x,\xi \right) =\left( 
	\begin{array}{cc}
		0 & -\kappa x^{T} \\ 
		x & C_{\xi }%
	\end{array}
	\right)
\end{equation*}%
($x$ is a column vector and $T$ denotes transpose). One has that $\mathfrak{g%
}_{1}=\mathfrak{so}\left( 4\right) $ and $\mathfrak{g}_{-1}=\mathfrak{o}%
\left( 1,3\right) $. We take $o=e_{0}$ and denote by $K_{\kappa }$ the
isotropy subgroup at this point and by $\mathfrak{k}_{\kappa }$ the Lie
algebra of $K_{\kappa }$.

Let $\mathfrak{g}_{\kappa }=\mathfrak{p}_{\kappa }\oplus \mathfrak{k}%
_{\kappa }$ be the Cartan decomposition of $\mathfrak{g}_{\kappa }$
associated with the point $o$, that is, $\mathfrak{p}_{\kappa }$ and $%
\mathfrak{k}_{\kappa }$ consist of the matrices $Z_{\kappa }\left(
x,0\right) $ and $Z_{\kappa }\left( 0,\xi \right) $, with $x,\xi \in 
\mathbb{R}^{3}$, respectively. The corresponding Cartan decomposition of $%
G_{\kappa }$ is given by $G_{\kappa }=\exp \left( \mathfrak{p}_{\kappa
}\right) K_{\kappa }$. Let $\pi :G_{\kappa }\rightarrow M_{\kappa }$, $\pi
\left( g\right) =g\left( o\right) .$

We consider on $G_{\kappa }$ the left invariant pseudo-Riemannian structure
given at the identity by%
\begin{equation}
	\left\langle Z_{\kappa }\left( x,\xi \right) ,Z_{\kappa }\left( y,\eta
	\right) \right\rangle =\tfrac{1}{4}\left( \left\langle x,\eta \right\rangle
	+\left\langle y,\xi \right\rangle \right) \text{,}  \label{33}
\end{equation}%
which in the Euclidean case coincides with that in (\ref{metricKappaCero}).
It has signature $\left( 3,3\right) $ and is bi-invariant. Indeed, since $%
G_{\kappa }$ is connected, it suffices to show that ad$_{Z}$ is skew
symmetric for all $Z\in \mathfrak{g}_{\kappa }$. This follows from this
expression for the Lie bracket:%
\begin{equation}
	\left[ Z_{\kappa }\left( x,\xi \right) ,Z_{\kappa }\left( y,\eta \right) %
	\right] =Z_{\kappa }\left( C_{\xi }y-C_{\eta }x,\kappa x\times y+\xi \times
	\eta \right) \text{,}  \label{corZZ}
\end{equation}%
which can be checked using the identities $yx^{T}-xy^{T}=C_{x\times y}=\left[
C_{x},C_{x}\right] $. Note that for $\kappa =\pm 1$, this metric is not the
one corresponding to the Killing form on the simple Lie algebra $\mathfrak{g}%
_{\kappa }$; the Killing form is not the only nondegenerate bi-invariant
form ($\mathfrak{g}_{1}$ is not simple and $\mathfrak{g}_{-1}$ is not absolutely simple, since it is the
complexification of $\mathfrak{o}\left(
1,2\right) $).

\medskip

The group $G_{\kappa }$ acts simply transitively on $SO\left( M_{k}\right) $%
. The action induces the bijection%
\begin{equation}
	\mathcal{I}:G_{\kappa }\rightarrow SO\left( M_{k}\right) \text{,\ \ \ \ \ } 
	\mathcal{I}\left( g\right) =\left( dg\right) _{o}\text{.}  \label{identif}
\end{equation}%
For $g,h\in G_{\kappa }$ we have $\left( dg\right) \left( \mathcal{I}\left(
h\right) \right) =\left( dg\right) \circ \left( dh\right) _{o}=d\left(
g\circ h\right) _{o}=d\left( L_{g}\left( h\right) \right) _{o}=\mathcal{I}%
\left( L_{g}\left( h\right) \right) $, where $L_{g}:G_{\kappa }\rightarrow
G_{\kappa }$ denotes left multiplication by $g$. Hence,%
\begin{equation}
	dg=\mathcal{L}_{g}=_{\text{def}}\mathcal{I}\circ L_{g}\circ \mathcal{I}
	^{-1}:SO\left( M_{\kappa }\right) \rightarrow SO\left( M_{\kappa }\right) 
	\text{.}  \label{dgLg}
\end{equation}

We consider on $SO\left( M_{\kappa }\right) $ the pseudo-Riemannian metric
induced from that on $G_{\kappa }$ in (\ref{33}) by the identification $%
\mathcal{I}$ in (\ref{identif}).

The next proposition provides a geometrical meaning of the pseudo-Riemannian
metric (\ref{metricKappaCero}) on $\mathbb{R}^{3}\times SO_{3}$ considered
in the introduction and moreover for the metric (\ref{33}) on $G_{\kappa }$
above. Theorem \ref{main} thereby acquires significance in relation to
positive frame vorticity.

	\begin{proposition}
	\label{helicitySL}Let $U$ be an open subset of $M_{\kappa }$. A section $%
	b:U\rightarrow SO\left( U\right) $ has positive  vorticity if and only if it
	determines a space-like submanifold of $SO\left( M_{\kappa }\right) $.
	\end{proposition}

Before proving the proposition, we state the following lemma. We call $\iota
=$ id$_{T_{o}M_{\kappa }}$ and identify $T_{o}M_{\kappa }=e_{0}^{\bot }\cong 
\mathbb{R}^{3}$. For simplicity, in the proofs we sometimes omit the
subindex $\kappa $.

\begin{lemma}
	\label{LemaZ}Let $b:U\rightarrow SO\left( M_{\kappa }\right) $ be a local
	section of $SO\left( M_{\kappa }\right) \rightarrow M_{\kappa }$.
	
	\smallskip
	
	\emph{a)} If $o\in U$, $b\left( o\right) =\iota $ and $x\in T_{o}M_{\kappa }$, then 
	$\left( db\right) _{o}\left( x\right) =d\mathcal{I}_{I_{6}}\left( Z_{\kappa
	}\left( x,x^{b}\right) \right)$. 
	
	\smallskip
	
	\emph{b)} Let $p\in U$ and suppose that $b\left( p\right) =\mathcal{I}\left(
	g\right) $, with $g\in G$ (in particular, $g\left( o\right) =p$). Let $\bar{b%
	}=\left( dg\right) ^{-1}\circ b\circ g$, which is a section from $%
	g^{-1}\left( U\right) $ to $SO\left( M_{\kappa }\right) $ with $\bar{b}%
	\left( o\right) =\iota $. Then%
	\begin{equation}
		\left( db\right) _{p}\left( y\right) =\left( d\mathcal{L}_{g}\right) _{\iota
		}\left( d\bar{b}_{o}\left( x\right) \right) \text{\ \ \ \ \ \ \ \ and\ \ \ \
			\ \ }y^{b}=\left( dg\right) _{o}\left( x^{\bar{b}}\right)   \label{dbyb}
	\end{equation}%
	for $x\in T_{o}M_{\kappa }$ and $y=dg_{o}\left( x\right) $.
\end{lemma}

\begin{proof}
	We call $\sigma \left( t\right) =\exp \left( tZ\left( x,0\right) \right) $,
	a curve in $G$. Let $\gamma $ be the geodesic in $M$ with $\gamma \left(
	0\right) =o=e_{o}$ and $\gamma ^{\prime }\left( 0\right) =x$. It is well
	known that $\gamma \left( t\right) =\pi \left( \sigma \left( t\right)
	\right) =\sigma \left( t\right) \left( o\right) $. By (\ref{identif}), 
	\begin{equation*}
		b\left( \gamma \left( t\right) \right) =\mathcal{I}\left( g\left( t\right)
		\right)
	\end{equation*}%
	for some curve $g$ in $G$. Hence $\left( db\right) _{o}\left( x\right) =d%
	\mathcal{I}_{I_{6}}\left( g^{\prime }\left( 0\right) \right) $.
	
	Suppose that the Cartan factorization of $g\left( t\right) $ is given by $%
	g\left( t\right) =\exp \left( X\left( t\right) \right) k\left( t\right) $,
	with $X\left( t\right) \in \mathfrak{p}$ and $k\left( t\right) \in K$.
	Evaluating at $t=0$ we get $g\left( 0\right) =I_{6}$ and so $X\left(
	0\right) =0$ and $k\left( 0\right) =I_{6}$.
	
	Differentiating $g$ at $t=0$ we obtain that $g^{\prime }\left( 0\right)
	=X^{\prime }\left( 0\right) +k^{\prime }\left( 0\right) $, which gives the
	Cartan decomposition of $g^{\prime }\left( 0\right) $. Hence, it suffices to
	show that 
	\begin{equation}
		X^{\prime }\left( 0\right) =Z\left( x,0\right) \text{\ \ \ \ \ \ \ and\ \ \
			\ \ \ \ \ }k^{\prime }\left( 0\right) =Z\left( 0,x^{b}\right) \text{.}
		\label{XprimaKprima}
	\end{equation}
	
	We have $\pi \left( \exp \left( X\left( t\right) \right) \right) =\pi \left(
	g\left( t\right) \right) =\gamma \left( t\right) $. Since $\left. \pi \circ
	\exp \right\vert _{\mathfrak{p}}:\mathfrak{p}\rightarrow M$ is a local
	diffeomorphism near $0\in \mathfrak{p}$, we have that $X\left( t\right)
	=tZ\left( x,0\right) $ and so the first identity in (\ref{XprimaKprima})
	holds. Also, 
	\begin{equation*}
		b\left( \gamma \left( t\right) \right) =\mathcal{I}\left( g\left( t\right)
		\right) =\mathcal{I}\left( \sigma \left( t\right) k\left( t\right) \right)
		=\left( d\sigma \left( t\right) \right) _{o}\left( dk\left( t\right) \right)
		_{o}\text{.}
	\end{equation*}
	
	It is well known that $\left( d\sigma \left( t\right) \right) _{o}$ realizes
	the parallel transport along $\gamma $ between $0$ and $t$. Thus, we have 
	\begin{equation*}
		C_{x^{b}}\left( v\right) =\left( \nabla _{x}b\right) \left( v\right) =\left. 
		\frac{D}{dt}\right\vert _{0}b\left( \gamma \left( t\right) \right) \left(
		v\right) =\left. \frac{D}{dt}\right\vert _{0}\left( dk\left( t\right)
		\right) _{o}\left( v\right) =k^{\prime }\left( 0\right) \left( v\right)
	\end{equation*}%
	for all $v\in T_{o}M\cong \mathbb{R}^{3}$. This implies the validity of the
	second identity in (\ref{XprimaKprima}).
	
	\medskip
	
	b) It is a direct consequence of (a) using the invariance by the action of $%
	G$. For the sake of completeness we present the computations. By
	the hypothesis and (\ref{dgLg}) we have that 
	\begin{equation*}
		db\circ dg=d\left( dg\right) \circ d\bar{b}=d\mathcal{L}_{g}\circ d\bar{b}%
		\text{.}
	\end{equation*}%
	Evaluating at $x$ we get the first expression in (\ref{dbyb}). The second
	one follows from the identities defining $y^{b}$ and $x^{\bar{b}}$ and the
	standard facts that $\left( dg\right) \circ C_{u}=C_{dg\left( u\right)
	}\circ \left( dg\right) $ and $\nabla _{v}\left( dg\circ \bar{b}\circ g^{-1}\right) =\left( dg\right) \circ
		\left( \nabla _{u}\bar{b}\right) \circ g^{-1}$ 
		for $v=dg\left( u\right) $, since $g$ is a direct isometry of $M$.
\end{proof}


\begin{proof}[Proof of Proposition \protect\ref{helicitySL}]
	Let $b:U\rightarrow SO\left( U\right) $ be a section and let $p\in U$. We
	will show that the tangent space of the submanifold determined by $b$ at $%
	b\left( p\right) $ is space-like if and only if $b$ has positive vorticity at 
	$p$. It suffices to see that, given $y\in T_{p}M_{\kappa }$, $\left\Vert
	\left( db\right) _{p}\left( y\right) \right\Vert >0$ if and only if $%
	\left\langle y^{b},y\right\rangle >0$.
	
	Suppose that $b\left( p\right) =\mathcal{I}\left( g\right) $ with $g\in
	G_{\kappa }$. Let $\bar{b}=\left( dg\right) ^{-1}\circ b\circ g$, which is a
	section $g^{-1}\left( U\right) \rightarrow SO\left( M_{\kappa }\right) $
	with $\bar{b}\left( o\right) =\iota $. Using the notation and assertions of
	Lemma \ref{LemaZ}, we write $y=\left( dg\right) _{o}\left( x\right) $ with $%
	x\in T_{o}M_{\kappa }$ and have that 
	\begin{equation*}
		\left\Vert \left( db\right) _{p}\left( y\right) \right\Vert =\left\Vert d%
		\bar{b}_{o}\left( x\right) \right\Vert =\left\Vert Z_{\kappa }\left(
		x,x^{b}\right) \right\Vert =\tfrac{1}{2}\left\langle x,x^{b}\right\rangle =%
		\tfrac{1}{2}\left\langle y,y^{\bar{b}}\right\rangle \text{.}
	\end{equation*}%
	The first and last equalities follow from part (b) of that lemma (since $%
	\mathcal{L}_{g}$ is an isometry of $SO\left( M_{\kappa }\right) $ for the
	metric in (\ref{33})) and the second one follows from part (a), since $%
	\mathcal{I}$ is an isometry.
\end{proof}


\begin{remark}
The local sections $\varphi $ of $\mathbb{R}^{3}\rtimes SO_{3}
$ corresponding to the spherical shells $\left\{ v\in \mathbb{R}^{3}\mid
\left( 2k-1\right) \pi <\left\vert v\right\vert <2k\pi \right\} $ with $k\in 
\mathbb{N}$ in Theorem \ref{main} have directions of negative vorticity at any point.
\end{remark}

\section{The intrinsic geometry of the  submanifold $\protect\phi %
	\left( B\right) $\label{ig}}

We present some properties of the submanifold $\phi :B=\left\{ v\in \mathbb{R%
}^{3}\mid \left\vert v\right\vert <\pi \right\} \rightarrow \mathbb{R}%
^{3}\rtimes S^{3}$ of Theorem \ref{mainLIft}, endowed with the induced
Riemannian metric.

\begin{proposition}
	For any unit vector $u\in \mathbb{R}^{3}$, the radial curve $r\mapsto \phi
	\left( ru\right) $ is the reparametrization of an inextensible geodesic in $%
	\phi \left( B\right) $ of finite length. In particular, $\phi \left(
	B\right) $ is not complete.
	
	The metric completion of $\phi \left( B\right) $ is $\left( \phi \left(
	B\right) \cup \left\{ \ast \right\} ,d\right) $, obtained by adding a point $%
	\ast $, and the metric topology is the one-point compactification of $\phi
	\left( B\right) $, rendering it homeomorphic to $S^{3}$. The distance $d$ is
	not Riemannian.
\end{proposition}

\begin{proof}
	It is convenient to work with the Riemannian metric on $B$ induced by $\phi $%
	, which we call $g$. Notice that $S^{3}$ acts on $\left( B,g\right) $ by
	isometries via $\left( q,x\right) \mapsto I_{q}\left(x\right)$, by the rotational
	symmetry stated in Proposition \ref{PiIso}.
	
	Given a unit vector $w\in \mathbb{R}^{3}$, the intersection with $B$ of the
	line $\mathbb{R}w$ through the origin is the connected component the set of
	fixed points of an isometry of $B$, for instance, any nontrivial rotation
	around $\mathbb{R}w.$ Then it is the image of a geodesic in $\left(
	B,g\right) $ (see, for instance, Proposition 10.3.6 in \cite{BCO}). Hence,
	the arc length reparametrization with respect to $g$ of $\alpha \left(
	r\right) =rw$ is a geodesic. We compute 
	\begin{equation*}
		\left\vert \alpha ^{\prime }\left( r\right) \right\vert ^{2}=\left\Vert
		\left( \phi \circ \alpha \right) ^{\prime }\left( r\right) \right\Vert
		=\left\Vert d\phi _{rw}\left( w\right) \right\Vert =\tfrac{1}{2}\,\ell
		^{\prime }\left( r\right)
	\end{equation*}%
	(we have used (\ref{dphii}) with $\theta \left( r\right) =r/2$; the
	expression is valid with $w$ instead of $i$ due to the rotational symmetry).
	Now, let 
	\begin{equation*}
		\sigma :\left( -\pi ,\pi \right) \rightarrow \mathbb{R}\text{,\ \ \ \ \ \ }%
		\sigma \left( r\right) =\int_{0}^{r}\tfrac{1}{\sqrt{2}}\left( \ell ^{\prime
		}\left( t\right) \right) ^{1/2}~dt
	\end{equation*}%
	be the signed arc length of $\alpha $. We call $L=\sigma \left( \pi \right) $%
	, which is finite, since $\ell ^{\prime }\ $is bounded on $\left( -\pi ,\pi
	\right) $ (see (\ref{lPrime})). Hence $\left( -L,L\right) \rightarrow B$, $%
	t\mapsto \sigma ^{-1}\left( t\right) w$ is an inextensible geodesic of $%
	\left( B,g\right) $. Therefore, the geodesic exponential map of $\left(
	B,g\right) $ at $0$ is given by 
	\begin{equation*}
		\text{Exp}_{0}:\left\{ v\in \mathbb{R}^{3}\mid \left\vert v\right\vert
		<L\right\} \rightarrow \left( B,g\right) \text{,\ \ \ \ \ \ Exp}_{0}\left(
		tw\right) =\sigma ^{-1}\left( t\right) w
	\end{equation*}%
	and it is moreover a diffeomorphism. In particular, $B$ is a normal ball
	centered at the origin and for any $0<r<\pi $ the sphere $S_{r}=\left\{ v\in
	B\mid \left\vert v\right\vert =r\right\} $ is a geodesic sphere, which is
	round, since it is preserved by the action of $S^{3}$.
	
	Let $\left( N,d\right) $ be the completion of $\left( B,g\right) $. We
	consider equivalences classes of Cauchy sequences in $\left( B,g\right) $.
	In order to show that $N$ is the one-point compactification of $B$, it
	suffices to verify that given sequences $r_{n}u_{n}$, $s_{n}v_{n}$ in $B$
	with $\left\vert u_{n}\right\vert =\left\vert v_{n}\right\vert =1$ and $%
	\lim_{n\rightarrow \infty }r_{n}=\lim_{n\rightarrow \infty }s_{n}=\pi $,
	then $d_{g}\left( r_{n}u_{n},s_{n}v_{n}\right) \rightarrow 0$ as $%
	n\rightarrow \infty $.
	
	Suppose that $0\leq s\leq r<\pi $ and $\left\vert u\right\vert =\left\vert
	v\right\vert =1$. Consider the piecewise smooth curve obtained by
	juxtaposing parametrizations of the radial segment joining $su$ with $ru$
	and an arc of circle joining $ru$ with $rv$. The segment has $g$-length 
	\begin{equation}
		\sigma ^{-1}\left( r\right) -\sigma ^{-1}\left( s\right) =\int_{s}^{r}\tfrac{%
			1}{\sqrt{2}}\left( \ell ^{\prime }\left( t\right) \right) ^{1/2}~dt\text{.}
		\label{eq1}
	\end{equation}%
	Also, by (\ref{dphiu}) with $~l=\ell $ and $\theta \left( r\right) =r/2$,
	the $g$-length of any Euclidean great circle in $S_{r}$ is 
	\begin{equation}
		2\pi r\left\Vert d\phi _{ri}\left( z\right) \right\Vert ^{1/2}=\sqrt{2}~\pi
		\left( \ell \left( r\right) \sin r\right) ^{1/2}  \label{eq2}
	\end{equation}%
	(any unit $z\perp i$). Since both (\ref{eq1}) and (\ref{eq2}) tend to $0$ as 
	$s\rightarrow \pi ^{-}$, we have that $d_{g}\left(
	r_{n}u_{n},s_{n}v_{n}\right) \rightarrow 0$ as $n\rightarrow \infty $.
	
	For $0\leq r<\pi $, $S_{r}$ defined above is the $d$-sphere in $N$ of radius 
	$\sigma \left( \pi \right) -\sigma \left( r\right) $ centered at $\ast $. By
	(\ref{dphiu}) with $l=\ell $ and $\theta \left( r\right) =r/2$, its $g$-area
	is%
	\begin{equation*}
		4\pi r^{2}\left\vert d\phi _{ri}\left( z\right) \right\vert ^{2}=4\pi
		r^{2}\left\Vert d\varphi _{ri}\left( z\right) \right\Vert =2\pi \ell \left(
		r\right) \sin r\text{.}
	\end{equation*}
	
	Now, for $\rho >0$ let $A\left( \rho \right) $ be the area of the distance 
	sphere of radius $\rho $ centered at the point $\ast $. We compute 
	\begin{equation*}
		\lim_{\rho \rightarrow 0}\frac{A\left( \rho \right) }{4\pi \rho ^{2}}%
		=\lim_{r\rightarrow \pi ^{-}}\frac{g\text{-area~}\left( S_{r}\right) }{4\pi
			\left( \sigma \left( \pi \right) -\sigma \left( r\right) \right) ^{2}}%
		=\lim_{r\rightarrow \pi ^{-}}\frac{\ell \left( r\right) \sin r}{\left(
			\int_{r}^{\pi }\left( \ell ^{\prime }\left( t\right) \right)
			^{1/2}~dt\right) ^{2}}=\infty
	\end{equation*}%
	by L'Hospital's rule, since the derivatives of the numerator and the
	denominator tend to $-\ell \left( \pi \right) \neq 0$ and $0$ as $%
	r\rightarrow \pi ^{-}$, respectively. Then the metric is not Riemannian,
	since otherwise, the limit would have been equal to $1$ (see for instance 
	\cite{GV}).
\end{proof}

\medskip 

Let $\mathcal{S}$ be the boundary of $\phi \left( B\right) $ in $\mathbb{R}%
^{3}\rtimes S^{3}$, which is a two-sphere. By the proposition above, one may
think that it collapses to the point $\ast $ in $N$. The following
proposition should confirm the insight we gained in Section
2: On the one hand, the metric on $\mathbb{R}^{3}\rtimes S^{3}$ degenerates
completely on $\mathcal{S}$, and on the other hand, positive vorticity fails
for the section of $SO\left( \mathbb{R}^{3}\right) $ associated to $\Phi
=\Pi \circ \phi $ for directions tangent to the two-sphere $\pi \left( \Pi
\left( \mathcal{S}\right) \right) $ in $\mathbb{R}^{3}$.

\begin{proposition}
	The boundary $\mathcal{S}$ equals $\left\{ \left( \ell \left( \pi \right)
	u,u\right) \in \operatorname{Im}\mathbb{H}\times S^{3}\mid \left\vert u\right\vert
	=1\right\} $ and its projection to $\mathbb{R}^{3}\rtimes S^{3}$ is $\left\{
	\left( \ell \left( \pi \right) u,R_{\pi u}\right) \mid \left\vert
	u\right\vert =1\right\} $. Both are totally null.
	
	Let $b:\mathbb{R}^{3}\rightarrow SO\left( \mathbb{R}^{3}\right) $ be the
	section associated with $\Phi $, that is, $b\left( \ell \left( r\right)
	u\right) =R_{ru}$. Let $\partial B=\left\{ v\in \mathbb{R}^{3}\mid \left\vert v\right\vert
	=\ell \left( \pi \right) \right\} $ be the boundary of $B$. Then $\left\langle X^{b},X\right\rangle =0
	$ for any vector field $X$ tangent to $\partial B$. 
\end{proposition}

We recall an expression for the rotation $R_{x}$ in (\ref{Rx}) in
quaternionic terms that will be useful below: For any $0\neq x\in \mathbb{R}%
^{3}\cong \operatorname{Im}\mathbb{H}$ one has
\begin{equation}
	R_{x}=I\left( \cos \left( \left\vert x\right\vert /2\right) +\sin \left(
	\left\vert x\right\vert /2\right) x/\left\vert x\right\vert \right) \text{,}
	\label{RxQ}
\end{equation}%
with $I$ as defined before (\ref{Pi}).

\begin{proof}
	The boundary of $\phi \left( B\right) $ in $\operatorname{Im}\mathbb{H}\times S^{3}$
	equals $\left\{ \left( \ell \left( \pi \right) u,u\right) \mid \left\vert
	u\right\vert =1\right\} $, since $e^{\pi u/2}=u$. By (\ref{dphiu}) with $%
	r=\pi $, $l=\ell $ and $\theta \left( r\right) =r/2$, it is totally null (by
	the rotational symmetry we may consider $\pi i$ instead of an arbitrary
	element of  $\partial B$). 	
	The projection of $\mathcal{S}$ onto $\mathbb{R}^{3}\rtimes S^{3}$ is as
	stated, since for $\left\vert u\right\vert =1$, $I_{u}=R_{\pi u}$ holds by (%
	\ref{RxQ}). It is also totally null, since $\Pi $ is a local isometry.
	
	Let $v=\ell \left( \pi \right) u\in \partial B$, with $\left\vert u\right\vert =1$
	and let $x\in T_{v}\left(\partial B\right)=v^{\bot }$. We may suppose that $x=\ell \left( \pi
	\right) y$ with $\left\vert y\right\vert =1$, $y\perp u$. By definition of $%
	x^{b}$, 
	\begin{equation}
		\left( \nabla _{x}b\right) _{v}=C_{x^{b}}\circ b\left( v\right) \text{.}
		\label{xxb}
	\end{equation}
	
	In order to compute the left hand side we call $\alpha \left( t\right) =\cos
	t~u+\sin t~y$. Note that $\beta =\ell \left( \pi \right) \alpha $ satisfies $%
	\beta \left( 0\right) =v$ and $\beta ^{\prime }\left( 0\right) =x$. By (\ref%
	{RxQ}), 
	\begin{equation}
		b\left( \ell \left( \pi \right) \alpha \left( t\right) \right) \left(
		z\right) =R_{\pi \alpha \left( t\right) }\left( z\right) =\alpha \left(
		t\right) z\overline{\alpha \left( t\right) }\text{,}  \label{bR}
	\end{equation}%
	for $z\in \mathbb{R}^{3}$. Then we have%
	\begin{eqnarray}
		\left( \nabla _{x}b\right) _{v}\left( z\right)  &=&\left( \nabla _{\ell
			\left( \pi \right) y}b\right) \left( z\right) =\left. \frac{D}{dt}%
		\right\vert _{0}b\left( \ell \left( \pi \right) \alpha \left( t\right)
		\right) \left( z\right)   \label{yzu} \\
		&=&\left. \frac{d}{dt}\right\vert _{0}\alpha \left( t\right) z\overline{%
			\alpha \left( t\right) }=yz\left( -u\right) -uzy=-yzu-uzy\text{.}  \notag
	\end{eqnarray}
	
	Evaluating (\ref{bR}) at $t=0$ yields $b\left( v\right) \left( z\right) =uz%
	\bar{u}=-uzu$; in particular $\left( b\left( v\right) \right) ^{-1}=b\left(
	v\right) $ ($u^{2}=-1$). Consequently, by (\ref{xxb}), 
	\begin{equation*}
		C_{x^{b}}\left( w\right) =\left( \nabla _{x}b\right) _{v}\left( \left(
		b\left( v\right) \right) ^{-1}\left( w\right) \right) =\left( \nabla
		_{x}b\right) _{v}\left( -uwu\right) \text{.}
	\end{equation*}
	
	Using (\ref{yzu}) and properties of quaternions, such as $u^{2}=-1$ and $%
	-yu=uy=u\times y$ since $\left\{ y,u\right\} $ is orthonormal, we have%
	\begin{equation*}
		C_{x^{b}}\left( w\right) =uyw-wuy=2\left( u\times y\right) \times w=C_{2\left( u\times y\right)
		}\left( w\right) 
	\end{equation*}%
	for all $w$. Hence, $x^{b}=2\left( u\times y\right) $, which is orthogonal to 
	$x$.
\end{proof}

\section{Optimal frame vorticity\label{oh}}

\begin{definition}
	Let $M$ be an oriented three-dimensional Riemannian manifold. A section $b$
	of $SO\left( M\right) \rightarrow M$ is said to have \textbf{optimal vorticity} 
	if for any vector field $X$ on $M$, the associated vector field $%
	X^{b}$ satisfies $X^{b}=X$.
\end{definition}

We will show below that $SO\left( S^{3}\right) \rightarrow S^{3}$ admits a
global section with optimal vorticity.

Recall that $P:S^{3}\times S^{3}\rightarrow SO_{4},$ $P\left( p,q\right)
\rightarrow L_{p}\circ R_{\bar{q}}$, is a surjective two-to-one morphism,
where $L_{p},R_{p}:G\rightarrow G$ denote left and right multiplication by $p
$, respectively.

\begin{lemma}
	\label{S3xS3so4}Let $G_{1}=SO_{4}$ and $S^{3}\times S^{3}$ be endowed with
	the split pseudo-Riemannian metrics given by \emph{(\ref{33})} and 
	\begin{equation}
		\left\langle \left( x,y\right) ,\left( x^{\prime },y^{\prime }\right)
		\right\rangle =\tfrac{1}{2}\left( \left\langle x,x^{\prime }\right\rangle
		-\left\langle y,y^{\prime }\right\rangle \right) \text{,}  \label{splitS3xS3}
	\end{equation}%
	for $x,x^{\prime }\in T_{p}S^{3}$, $y,y^{\prime }\in T_{q}S^{3}$,
	respectively. Then $P$ is a local isometry.
\end{lemma}

\begin{proof}
	Since both metrics are invariant and $P$ is a surjective morphism, it
	suffices to show that $dP_{\left( 1,1\right) }:T_{1}S^{3}\times T_{1}S^{3}=%
	\operatorname{Im}\mathbb{H}\times \operatorname{Im}\mathbb{H}\rightarrow \mathfrak{so}_{4}$
	is a linear isometry. Let $z\in \mathbb{R}^{4}\cong \mathbb{H}$. For $v,w\in 
	\operatorname{Im}\mathbb{H}$ we compute 
	\begin{equation*}
		\left( dP_{\left( 1,1\right) }\left( v,w\right) \right) \left( z\right)
		=\left. \frac{d}{dt}\right\vert _{0}e^{tv}ze^{-tw}=vz-zw=\left( \ell
		_{v}-\rho _{w}\right) \left( z\right) \text{,}
	\end{equation*}%
	where $\ell _{v}$, $\rho _{v}$ denote left (respectively, right)
	multiplication by $v$ on $\mathbb{H}$. We verify that 
	\begin{equation}
		\ell _{v}=Z_{1}\left( v,v\right) \text{\ \ \ \ \ \ and\ \ \ \ \ \ \ }\rho
		_{v}=Z_{1}\left( v,-v\right) .  \label{lv}
	\end{equation}%
	Indeed, if $a\in \mathbb{R}$ and $u\in \operatorname{Im}\mathbb{H}$, we have 
	\begin{equation*}
		\left( 
		\begin{array}{cc}
			0 & -v^{T} \\ 
			v & C_{v}%
		\end{array}%
		\right) \left( 
		\begin{array}{c}
			a \\ 
			u%
		\end{array}%
		\right) =\left( 
		\begin{array}{c}
			-\left\langle v,u\right\rangle  \\ 
			av+C_{v}u%
		\end{array}%
		\right) \text{.}
	\end{equation*}%
	Since $v\left( a+u\right) =av+vu=av-\left\langle v,u\right\rangle +v\times u$%
	, the first identity in (\ref{lv}) follows. The second one can be checked in
	a similar manner. Hence,%
	\begin{equation*}
		dP_{\left( 1,1\right) }\left( v,w\right) =Z_{1}\left( v,v\right)
		-Z_{1}\left( w,-w\right) =Z\left( v-w,v+w\right) \text{.}
	\end{equation*}%
	Since $2\,\frac{1}{4}\left\langle
	v-w,v+w\right\rangle =\frac{1}{2}\left( \left\vert v\right\vert
	^{2}-\left\vert w\right\vert ^{2}\right) $, (\ref{33}) implies that $P$ is a local isometry.
\end{proof}

\begin{theorem}
	The section $b$ of $SO\left( S^{3}\right) \rightarrow S^{3}$ given by $%
	b\left( p\right) =\left( dL_{p}\right) _{1}$ has optimal vorticity.
\end{theorem}

\begin{proof}
	Let $X$ be a vector field on $S^{3}$ and $p\in S^{3}$. We want to see that $%
	\nabla _{X_{p}}b=C_{X_{p}}\circ b$. We consider first the case $p=1$ and
	call $x=X_{1}$. Let p $:\mathbb{H}\rightarrow T_{1}S^{3}$ be the orthogonal
	projection. For $z\in T_{1}S^{3}=\operatorname{Im}\left( \mathbb{H}\right) $ we
	compute 
	\begin{eqnarray*}
		\left( \nabla _{x}b\right) \left( z\right)  &=&\nabla _{x}\left( q\mapsto
		\left( dL_{q}\right) _{1}\left( z\right) \right) =\left. \frac{D}{dt}%
		\right\vert _{0}\left( dL_{e^{tx}}\right) _{1}\left( z\right) =\text{p~}%
		\left( \left. \frac{d}{dt}\right\vert _{0}e^{tx}z\right)  \\
		&=&xz-\left\langle xz,1\right\rangle =xz-\operatorname{Re}\left( xz\right) =x\times
		z=C_{x}\left( z\right) \text{.}
	\end{eqnarray*}%
	Hence $x^{b}=x$. Now, we deal with the case when $p$ is arbitrary in $S^{3}$%
	. Suppose $X_{p}=\left( dL_{p}\right) _{1}\left( x\right) $. Then $\gamma
	\left( t\right) =pe^{tx}$ has initial velocity $X_{p}$. For $z\in \operatorname{Im}%
	\mathbb{H}$ we compute 
	\begin{eqnarray*}
		\left( \nabla _{X_{p}}b\right) \left( z\right)  &=&\left. \frac{D}{dt}%
		\right\vert _{0}\left( dL_{pe^{tx}}\right) _{1}\left( z\right) =\left(
		dL_{p}\right) _{1}\left. \frac{D}{dt}\right\vert _{0}\left(
		dL_{e^{tx}}\right) _{1}\left( z\right)  \\
		&=&\left( dL_{p}\right) _{1}\left. \frac{D}{dt}\right\vert _{0}b\left(
		e^{tx}\right) \left( z\right) =\left( dL_{p}\right) _{1}\left( \nabla
		_{x}b\right) \left( z\right)  \\
		&=&\left( dL_{p}\right) _{1}C_{x}\left( z\right) =C_{\left( dL_{p}\right)
			_{1}x}\left( \left( dL_{p}\right) _{1}\left( z\right) \right)
		=C_{X_{p}}\left( b\left( z\right) \right) \text{.}
	\end{eqnarray*}%
	Consequently, $X^{b}=X$, as desired.
\end{proof}

Not surprisingly, the section $b$ in the previous theorem is also
distinguished concerning the volume:

\begin{theorem}
	\label{common}The submanifold of $SO\left( S^{3}\right) $ determined by the
	section $b$ of $SO\left( S^{3}\right) \rightarrow S^{3}$ given by $b\left(
	p\right) =\left( dL_{p}\right) _{1}$ is space-like and homologically volume
	maximizing.
\end{theorem}

\begin{remark} \rm 
	Compare with the calibration on the double covering of $SO\left( S^{3}\right) 
	$ in the Riemannian setting in \cite{JDG}.
\end{remark}

\begin{proof}
	By Lemma \ref{S3xS3so4}, $\mathcal{I}\circ P:S^{3}\times S^{3}\rightarrow
	SO\left( S^{3}\right) $ is a double covering and a local isometry with
	respect to the metrics (\ref{splitS3xS3}) and the one induced on $SO\left(
	S^{3}\right) $ from (\ref{33}) via $\mathcal{I}$ (with $\kappa=0$). Since $\mathcal{I}\circ P$
	sends $S^{3}\times \left\{ 1\right\} $ to the image of the section $b$, it
	suffices to show that $S^{3}\times \left\{ 1\right\} $ is homologically volume
	maximizing. Let $\theta $ be the bi-invariant $3$-form on $S^{3}$ such that $%
	\theta _{1}\left( i,j,k\right) =2^{-3/2}$ and let $p:S^{3}\times
	S^{3}\rightarrow S^{3}$ be the projection onto the first factor.
	
	Now, $\left\{ \sqrt{2}\left( i,0\right) ,\sqrt{2}\left( j,0\right) ,\sqrt{2}%
	\left( k,0\right) ,\sqrt{2}\left( 0,i\right) ,\sqrt{2}\left( 0,j\right) ,%
	\sqrt{2}\left( 0,k\right) \right\} $ is a positively oriented orthonormal
	basis of $ T_{\left( 1,1\right)
	}\left( S^{3}\times S^{3}\right) $ for the split metric (\ref{splitS3xS3}).
	Let $\left\{ \alpha _{1},\dots ,\alpha _{6}\right\} $ be its dual basis. It
	is easy to see that $p^{\ast }\theta $ is a left invariant $3$-form on $%
	S^{3}\times S^{3}$ such that $\left( p^{\ast }\theta \right) _{\left(
		1,0\right) }=\alpha _{1}\wedge \alpha _{2}\wedge \alpha _{3}$.
	
	Since $p^{\ast }\theta $ is closed, it is a split one-point calibration by
	the remark after the Fundamental Lemma of Calibrations in \cite{Mealy}. The
	computations above show that it calibrates $S^{3}\times \left\{ 1\right\} $,
	as desired.
\end{proof}

\begin{proposition}
	For $\kappa =0,-1$, no local section of $SO\left( M_{\kappa }\right) $ has
	optimal vorticity.
\end{proposition}

\begin{proof}
	Suppose that $b:U\rightarrow SO\left( M_{\kappa }\right) $ is a local
	section of $SO\left( M_{\kappa }\right) \rightarrow M_{\kappa }$ with
	optimal vorticity. Let $\mathcal{D}$ be the left invariant distribution on $%
	G_{\kappa }$ defined at the identity by 
	\begin{equation*}
		\mathcal{D}_{I_{6}}=\left\{ Z_{\kappa }\left( x,x\right) \mid x\in \mathbb{R}%
		^{3}\right\} 
	\end{equation*}%
	and let $\mathcal{E}$ be the corresponding distribution on $SO\left(
	M_{\kappa }\right) $ via the identification $\mathcal{I}$ in (\ref{identif}%
	), that is, $d\mathcal{I}_{g}\left( \mathcal{D}_{g}\right) =\mathcal{E}_{%
		\mathcal{I}\left( g\right) }$ for $g\in G_{\kappa }$.
	
	Let us see first that $b\left( U\right) $ is an integral submanifold of the
	distribution $\mathcal{E}$. Let $p\in U$ and $y\in T_{p}U$. Let $g$, $\bar{b}
	$ and $x$ as in Lemma \ref{LemaZ}. By parts (b) and (a) of that lemma, 
	\begin{equation*}
		\left( db\right) _{p}\left( y\right) =\left( d\mathcal{L}_{g}\right) _{\iota
		}\left( d\bar{b}_{o}\left( x\right) \right) =\left( d\mathcal{L}_{g}\right)
		_{\iota }d\mathcal{I}_{I_{6}}\left( Z_{\kappa }\left( x,x^{b}\right) \right) 
		\text{,}
	\end{equation*}%
	which belongs to $\left( d\mathcal{L}_{g}\right) _{\iota }\mathcal{E}\left( 
	\bar{b}\left( o\right) \right) =\mathcal{E}\left( b\left( p\right) \right) $, 
	since $x^{b}=x$ by hypothesis. Hence $T_{b\left( p\right) }b\left(
	U\right)\allowbreak  =\mathcal{E}_{b\left( p\right) }$. Then $b\left( U\right) $ is an
	integral submanifold of the distribution $\mathcal{E}$ (both of the same
	dimension). We have reached a contradiction, since the distribution $%
	\mathcal{D}$ is nowhere involutive for $\kappa =0,-1$. In fact, it is left
	invariant and by (\ref{corZZ}), at the identity one has 
	\begin{equation*}
		\left[ Z_{\kappa }\left( x,x\right) ,Z_{\kappa }\left( y,y\right) \right]
		=Z_{\kappa }\left( 2x\times y,\left( \kappa +1\right) x\times y\right) \text{%
			,}
	\end{equation*}%
	which does not belong to $\mathcal{D}_{I_{6}}$ if $\kappa \neq 1$.
\end{proof}

\bigskip

\noindent {\sc famaf} (Universidad Nacional de C\'ordoba) and {\sc ciem} (Conicet) 

\smallskip

\noindent Ciudad Universitaria, {\sc x5000hua}  C\'{o}rdoba, Argentina

\smallskip

\noindent marcos.salvai@unc.edu.ar \ ({\sc orcid} 0000-0001-9900-7752)


\end{document}